\documentclass[12pt]{amsart}

\usepackage{graphicx}

\usepackage{amssymb} 
\usepackage[T1]{fontenc}

\usepackage{amsmath} 

\usepackage{amsthm} 
\usepackage{enumerate,color}
\definecolor{Green}{rgb}{0,0.5,0}

\usepackage{url}

\numberwithin{equation}{section}
\theoremstyle{plain}
\newtheorem{theorem}{Theorem}[subsection]
\newtheorem{lemma}[theorem]{Lemma}
\newtheorem{proposition}[theorem]{Proposition}

\newtheorem{cor}[theorem]{Corollary}

\newtheorem*{mainthm} {Main Theorem}

\theoremstyle{definition}
\newtheorem{definition}[theorem]{Definition}

\newtheorem*{acknowledgements}{Acknowledgements}
\theoremstyle{remark}
\newtheorem{remark}[theorem]{Remark}

\newtheorem{example}[theorem]{Example}
\newtheorem*{mainex}{Main Example}

\newcommand{\cj}{\mathbb{C}}
\newcommand{\rj}{\mathbb{R}}
\newcommand{\nj}{\mathbb{N}}
\newcommand{\R}{{\mathcal R}}
\newcommand{\K}{{\mathcal K}}
\newcommand{\D}{\overline{\mathbb{D}}}

\usepackage{hyperref}

\begin{document}

\title[Non-autonomous iteration in $\mathbb{C}$]{ Non-autonomous iteration of  polynomials in the complex plane}

\author[M. Kosek and M. Stawiska]{ Marta Kosek\\ Institute of Mathematics\\ Faculty of Mathematics and Computer Science\\ Jagiellonian University\\ \L ojasiewicza 6, 30-348 Krak\'ow, POLAND\\
E-mail:
 Marta.Kosek@im.uj.edu.pl\\
\and
Ma\l gorzata Stawiska\\
American Mathematical Society--Mathematical Reviews\\416 Fourth St., Ann Arbor, Michigan, USA\\
E-mail: stawiska@umich.edu
}

\date{{\color{Green}\today}
}

\maketitle

\renewcommand{\thefootnote}{}



\renewcommand{\thefootnote}{\arabic{footnote}}
\setcounter{footnote}{0}


\begin{abstract}
 We consider a  sequence $(p_n)_{n=1}^\infty$ of polynomials with uniformly bounded zeros and  $\deg p_1\geq 1$, $\deg p_n\geq 2$ for $n\geq 2$, satisfying  certain asymptotic conditions. We prove that  the function sequence  $\left(\frac{1}{\deg p_n\cdot...\cdot \deg p_1}\log^+|p_n\circ...\circ p_1|\right)_{n=1}^\infty$ is uniformly convergent in $\cj$. The   non-autonomous filled Julia set  $\mathcal{K}[(p_{n})_{n=1}^\infty]$ generated by the  polynomial
 sequence $(p_{n})_{n=1}^\infty$ is defined and  shown  to be compact
 and regular with respect to the Green function.  Our toy example is generated by $t_n=\frac{1}{2^{n-1}}T_n,\ n\in\{1,2,...\}$, where $T_n$ is the classical Chebyshev polynomial of degree $n$.
\end{abstract}

\maketitle


\section{Introduction}

Behavior of iterates  $p^n:=p^{\circ n}$ of a  polynomial $p: \mathbb{C} \longrightarrow \mathbb{C}$ has been intensely studied since the nineteenth century.  For an   introduction to the subject see e.g. \cite{CGbook}.  Much more recent (and much less advanced) is the study of sequences of maps of the type $p_n\circ...\circ p_1$ where each $p_n$, $n\in\{1,2,...\},$ is a polynomial, but not necessarily the same one. Under the name of "generalized iteration", the works  \cite{BB03} and \cite{B97} deal with such sequences when the coefficients of underlying polynomials   satisfy some conditions (see Subsection \ref{subsection: KW} for more details).   For a  recent  study of iteration of quadratic polynomials with random coefficients see  \cite{LZ} and \cite{LZ24}.

In this article we extend the dynamical study to sequences of polynomials   without necessarily assuming anything about their  coefficients. Instead, we work with sequences of polynomials exhibiting certain asymptotic behavior, called guided sequences.
 Our class includes not only   the family $\mathcal{B}$ studied in \cite{BB03} and \cite{B97}, but also  important  sequences
 of polynomials associated
with a (fixed) compact polynomially convex set in the complex
plane,   namely sequences satisfying the Kalm\'ar-Walsh
condition. Such sequences  have been widely studied in complex approximation and interpolation
theory, often in connection with logarithmic potential theory
(\cite{Walsh}, \cite{BBCL}, \cite{CSZ4}, \cite{BE}).    For examples
see Subsection~~\ref{subsection: KW}.

\begin{mainex}
Recall that   the classical Chebyshev polynomials  satisfy
$T_{d_1}\circ T_{d_2} =T_{d_2}\circ T_{d_1}=T_{d_1d_2}$. All zeros of $T_d$ belong to the segment $[-1,1]$.
The polynomial $t_d=\frac{1}{2^{d-1}}T_d$ is the  monic polynomial with minimal uniform norm (the minimal
polynomial) of degree $d$ on $[-1,1]$,~~so
$$\lim_{d \to \infty}\|t_d\|_{[-1,1]}^{1/d}= {\rm cap}([-1,1]) =\frac12 
$$ and
$$\lim_{d \to \infty}\frac{1}{d}\log|t_d(z)|=g_{[-1,1]}(z) -\log 2$$
locally uniformly in $\mathbb{C}\setminus [-1,1]$.  Here $g_{[-1,1]}$ is the complex Green function of the segment $[-1,1]$  and ${\rm cap}([-1,1])$ is its logarithmic capacity (for the definitions and properties of these notions, see Section \ref{s: GF}). In consequence,
$\lim_{d \to \infty}\frac{1}{d}\log|T_d(z)|=g_{[-1,1]}(z) $ locally uniformly in $\mathbb{C}\setminus [-1,1]$.

For a sequence of  integers   $(d_n)_{n=1}^\infty$ not smaller than 2  we  get    $ T_{d_n} \circ ...\circ T_{d_1}=T_{d_n...d_1}$, hence $(T_{d_ n} \circ ...\circ T_{d_1})_{n=1}^\infty$ is a subsequence of the sequence of all  classical Chebyshev polynomials
with increasing degrees and
\[
\lim_{n \to \infty}\frac{1}{d_n...d_1}\log \left|(T_{d_n}\circ...\circ T_{d_1})
(z)\right|=g_{[-1,1]}(z)
\]
locally uniformly in $\mathbb{C}\setminus [-1,1]$.

  By a more dynamical approach
  it was shown in \cite{Ko21}  that
\[
\lim_{n \to \infty}\frac{1}{d_n...d_1}\log^+ \left|(T_{d_n}\circ...\circ T_{d_1})(z)\right|=g_{[-1,1]}(z)
\]
 and this convergence is uniform in the whole complex plane.
\end{mainex}

The function $\frac{1}{d_n...d_1}\log^+ \left|(T_{d_n}\circ...\circ T_{d_1})\right|$ is the Green function of the preimage of the  closed  unit disk, $g_{(T_{d_n}\circ...\circ T_{d_1})^{-1}\left(
\D(0,1)\right)}$   (cf.  (\ref{e:Greenaprzeciwobrazu}) and Example \ref{ex: koloGreen}), and the convergence of  functions can be thought of as the convergence in Klimek's metric of the sets $(T_{d_n}\circ...\circ T_{d_1})^{-1}\left(
\D(0,1)\right)$ to $[-1,1]$. It is natural to ask whether the convergence still holds if
  $(T_n)_{n=1}^\infty$ is replaced by another sequence $(p_n)_{n=1}^\infty$ of polynomials (not necessarily associated with a specific compact set $K$) and $
\D(0,1)$ is replaced by an arbitrary  compact, polynomially convex regular set $E$. We show that this is indeed the case when the polynomials $p_n$ satisfy certain asymptotic condition and their zeros are uniformly bounded. Our main result is the following (see Theorem \ref{thm:zbiornieautonomiczny}):

\begin{mainthm}\label{thm: firstmain}
Consider a sequence of polynomials $(p_n)_{n=1}^\infty$   with  $\deg p_1\geq 1$ and $\deg p_n\geq 2$ for $ n \geq 2$. Define $\alpha:=\limsup_{n \to \infty}\frac{1}{{\rm deg} \ p_n}$. Assume that  there exist an $r>0$ and   a function $h:\cj\longrightarrow \rj$ satisfying $h(z)-\alpha\log|z|\longrightarrow \infty \text{ as } |z|\to \infty$ such that   \begin{enumerate}
\item[$(i)$] $\displaystyle
\bigcup_{n=1}^\infty p_n^{-1}(0) \subset
 \D(0,r)$;
 \item[$(ii)$] $\displaystyle \exists N >1\ \forall n\geq N\ \forall \zeta\in \cj\setminus \D(0,r):\; \frac1{\deg p_n}\log^+|p_n(\zeta)|>h(\zeta);$
 \item[$(iii)$]  $\displaystyle \sup_{n \geq 1} \frac{1}{\deg p_n} \log^+|p_n(z)|<+\infty$ on a non-empty open set $D \subset \mathbb{C}$.
\end{enumerate}
Then, for every compact, polynomially convex and regular set $E$,  the sequence  $$\left((p_{n}\circ...\circ p_{1})^{-1}(E)\right)_{n=1}^\infty$$ is convergent in Klimek's metric to the limit $\K[(p_{n})_{n=1}^\infty]$, independent of $E$.
In particular $\K[(p_{n})_{n=1}^\infty]$ is non-empty, compact, regular and polynomially convex.
\end{mainthm}

We call the set  $\K[(p_{n})_{n=1}^\infty]$ the non-autonomous Julia set generated by the sequence $(p_{n})_{n=1}^\infty$.  We establish basic properties of such Julia sets, showing that they are non-empty, compact, polynomially convex and regular. Our methods are primarily based on the approach developed in \cite{KK03} (see also \cite{KK18}).

We include  a  considerable 
amount of background material in  our presentation to make it as self-contained as possible.
Section \ref{s: GF} starts with preliminaries from logarithmic potential theory and complex approximation theory. Later, we introduce the notion of a guided sequence  (such sequences satisfy conditions  $(i)$ and $(ii)$   in the Main Theorem).
Theorem \ref{thm: guided} provides a characterization of guided sequences. Lemma \ref{prop: tnzez} and Corollary \ref{cor:przeciwobrazkola} sharpen some previously available results.
Examples of guided sequences appear in Subsection \ref{subsection: KW}. Among them are Kalm\'ar-Walsh sequences,  important in approximation theory, and sequences 
 from the class $\mathcal{B}$, previously studied in connection with
 generalized iteration.     Using concrete examples (see Remark \ref{rem:notinB} and Example \ref{ex: nonKW}) we discuss  relations between these classes and the class of
guided sequences. 
More new material appears in Section \ref{s: Julia}, including the definition of a non-autonomous filled Julia set, as well as proofs of   existence and properties of such sets. The approach using Klimek's metric is explained. The presentation ends with Subsection \ref{s: toy}, where we explore in more detail the case of the non-autonomous filled Julia set generated by $t_d=\frac{1}{2^{d-1}}T_d,\ d\in\{1,2,...\}$, that is, by Chebyshev   (minimal)  polynomials on the segment $[-1,1]$.


\section{The Green function, the logarithmic capacity  and asymptotic behavior of polynomial sequences}\label{s: GF}

\subsection{The Green function and related notions}

Let us first define the Green function and the logarithmic capacity of a set (mostly following \cite[Section II.4]{SaTo}).   Our definitions are not the most general ones, but  they still yield a notion of a regular set which agrees with the one considered in literature. In particular, finite sets are not regular.

\begin{definition}\label{defin: Greenfunction}
 Let $D \subset \mathbb{C}$ be an unbounded domain. Consider a function  $g_D:D \longrightarrow \mathbb{R}$ with the following properties:
\begin{enumerate}
\item $g_D$ is harmonic and positive in $D$;

\item $g_D(z)$ tends to $0$ as $z \to \partial D$;

\item $g_D(z)-\log |z|$ tends to a finite number $\gamma$ as $z \to \infty$.
\end{enumerate}
If such a function $g_D$ exists, it is called  the {\it Green function} of $D$ with pole at infinity.
\end{definition}
 It can be proved  that if  the Green function exists, it is unique.

Let $K\subset \cj$ be compact. Define the polynomially convex hull of $K$~~as $$\widehat{K}:=\left\{z \in \mathbb{C}\ |\; \forall p - \mbox{ polynomial: }\; |p(z)|\leq \|p\|_K\right\},$$ where $\|p\|_K:=\sup_{z\in K}|p(z)|$. Let $D_\infty^K$ denote the unbounded component of $\mathbb{C}\setminus K$. We have $\widehat{K}=\mathbb{C}\setminus D_\infty^K$. The set $K$ is called {\it polynomially convex} if $K=\widehat{K}$.  Note that  polynomial convexity of $K$  is equivalent to the connectedness of $\cj\setminus K$.

\begin{definition}\label{def:g_K}Let $K\subset \cj$ be  a compact set such that the Green function $g_{D_\infty^K}$ of $D_\infty^K$ with pole at infinity exists. We say then that $K$ is {\it regular} and define {\it the Green function $g_K:\cj\longrightarrow \rj$ of $K$} via formula
$$
g_K(z)=\begin{cases}
         g_{D_\infty^K}(z), & \mbox{ if } z\in  {D_\infty^K}\\
         0, & \text{ if } z\in \widehat{K}
       \end{cases}.
$$ The number  $\exp(-\gamma)$, where $\gamma$ is the limit from
Definition \ref{defin: Greenfunction}(3),  is
called   the {\it logarithmic capacity} of $K$ and denoted by ${\rm cap}(K)$.
\end{definition}

If $K$ is regular, then $g_K=g_{\widehat{K}}$ and it is a continuous function in view of Definitions \ref{defin: Greenfunction} and \ref{def:g_K}. Moreover $g_K$ is subharmonic in $\cj$.

\begin{example} \label{ex: koloGreen} (\cite[2.6]{Sic1981}) Let $a \in \mathbb{C}$, $R>0$ and let $K=\overline{\mathbb{D}}(a,R):=\{z: |z-a|\leq R\}$. Then $g_K(z)=\log^+(|z-a|/R):=\max\{0, \log (|z-a|/R)\}$ and ${\rm cap}(K)=R$.
\end{example}

\begin{example} \label{ex: odcGreen} (\cite[\S 1.2.b)]{PS})
$g_{[-1,1]}(z)=\log\left|z+\sqrt{z^2-1}\right|$, with the square root branch defined in $\mathbb{C}\setminus (-\infty,0)$ so that $\sqrt{1}=1$.  
\end{example}

\begin{example} \label{ex: elipsyGreen} (\cite[\S1.2.c)]{PS}) Fix an $R>1$. Let $K=E_R$ be the ellipse with foci $-1,+1$ and semiaxes $a=\frac{1}{2}\bigl(R+\frac{1}{R}\bigr), \ b= \frac{1}{2}\bigl(R-\frac{1}{R}\bigr)$. Then $g_K(z)=\log^+\left(
{\left|z+\sqrt{z^2-1}\right|}/{R}\right)$ and ${\rm cap}(K)=(a+b)/2=R/2$.
\end{example}

Let us note also an easy consequence of the definitions.
 If $K\subset \cj$ is compact and regular, then
\begin{equation}\label{e:L+} g_K \in \mathcal{L}^+,
 \end{equation}
where $\mathcal{L}^+$ is a special  class of functions, defined below.
\begin{definition}[{cf. \cite[Introduction]{Sic1981}}] Let $SH(\cj)$ denote the family of all subharmonic functions on $\cj$. We define the {\it Lelong class}
\[
\mathcal{L}:=\left\{u\in SH(\cj)\ \left|\quad \exists C=C(u) \ \forall z \in \mathbb{C}: \ u(z) \leq C+\log^+|z|\right.\right\}.
\]
An important subclass of the Lelong class is
\[
\mathcal{L}^+:=\left\{u \in \mathcal{L}\ \left|\quad \exists c=c(u)\  \forall z \in \mathbb{C}: \ c+\log^+|z| \leq u(z)\right.\right\}.
\]
\end{definition}

We will  use the following result:
\begin{proposition}[{\cite[Lemma 3.4, (4), (5) and (6)]{Sic1981}}]
\label{prop: boundedabove} \label{prop: u*inL} Let $\mathcal{U} \subset \mathcal{L}$ be a non-empty family such that each function $v \in \mathcal{U}$ is continuous.  Let $u:=\sup\{v: v \in \mathcal{U}\}$. The following  are equivalent:
\begin{enumerate}
\item[$(i)$] There exists a non-empty open  $D \subset  \mathbb{C}$ such that $u(z) <+\infty$ for all $z \in D$;
\item[$(ii)$] $u$ is bounded from above on every compact set $K\subset \mathbb{C}$ (i.e., the family $\mathcal{U}$ is locally uniformly bounded from above).
  \item[$(iii)$] $u^* \in \mathcal{L}$, where $u^*(z):=\limsup_{\zeta \to z}u(\zeta)$.
\end{enumerate}
\end{proposition}

 The Green function of a compact regular set $K \subset \mathbb{C}$ can be also characterized 
 (cf. \cite[the beginning of Chapter 5 and Theorem 5.1.7]{Klimekbook}) for   every  $z\in \mathbb{C}$  as follows:
\begin{align}\label{e:greenchar}
g_K (z) =\sup \left\{u(z):\ u \in \mathcal{L}, u|_K \leq 0\right\}=\\ = \sup\left\{\frac{1}{{\rm deg} P}\log|P(z)|:\ P\text{ a  non-constant polynomial}, \|P\|_K \leq 1\right\}\nonumber
\end{align}

 Immediately from this characterization one can deduce  that if $A, B\subset \cj$ are compact and regular, then
\begin{equation}\label{e:AsubsetB}
A\subset B\quad \Longrightarrow   \quad g_A \geq g_B.
\end{equation}
Also, the following polynomial transformation formula can be  deduced from (\ref{e:greenchar}). Let $f$ be a polynomial of degree $n \geq 1$. Then
\begin{equation}\label{e:Greenaprzeciwobrazu}
  g_{f^{-1}(A)}=\frac1n g_A\circ f.
\end{equation}
In particular,
the preimage under $f$ of a  regular set, e.g.,  a closed disk, a line segment or an ellipse, is a regular set.
 \begin{definition}\label{def:sublevel}
Let $\varepsilon>0$
and let $K\subset \cj$ be compact and regular.  Then the   $\varepsilon$-sublevel set  of $g_K$ (also called the {\it $\varepsilon$-augmentation} of $K$) is $K_\varepsilon:=\left\{z\in \cj:\; g_K(z)\leq \varepsilon\right\}$.
\end{definition}
 If in addition  $K$ is  polynomially convex,  the family $\{K_\varepsilon\}_{\varepsilon >0}$ forms a neighbourhood base of the set $K$ in $\cj$ as shown in \cite[Corollary 1]{Kl95}.  It was proved by M. Mazurek (published in  \cite[Proposition 5.11]{Sic1981})  that
\begin{equation}\label{e:Mazurek}
  g_{K_\varepsilon}=\max(0,g_K-\varepsilon).
\end{equation}
  This result implies the following properties of the sublevel sets:

\begin{proposition}[{cf. \cite[Proposition 2.3]{BCKS}}]\label{prop:sublevel}   Let $K$ be a regular compact subset of $\cj$. Then
  for every $\varepsilon >0$ the set $K_\varepsilon$ is polynomially convex and regular. Furthermore,
 ${\rm cap} (K_\varepsilon)=\exp{(-\gamma+\varepsilon)}  =\exp(\varepsilon){\rm cap}(K)$.
\end{proposition}


\subsection{Properties of sequences of  polynomials}  Let us start with a property of   a single polynomial.
\begin{lemma}\label{lem: theta}
  If $P:\cj\longrightarrow \cj$ is a polynomial  and  $\deg P\geq 2$, then
  \begin{equation}\label{e:theta}
    \forall \theta>1 \ \exists R=R(\theta)>0: \quad |z|\geq R\quad \Longrightarrow \quad |P(z)|\geq \theta |z|.
  \end{equation}
\end{lemma}
\begin{proof} Observe that
${|P(z)|}/{|z|}\longrightarrow \infty,$ if $|z|\rightarrow \infty$, since $\deg P\geq 2$.
\end{proof}

Now we turn to sequences. First
we propose a modification of
\cite[Definition 3.1]{CHPP}:
\begin{definition} \label{def:K-guided} Let $K \subset \mathbb{C}$ be a    regular compact set. Let further $(p_n)_{n=1}^\infty$ be a sequence of complex polynomials such that $\deg p_n =n$.
The sequence $(p_n)_{n=1}^\infty$ is called {\it $K$-guided} if
there exist  $\delta>0, a>0$ and $b\in\rj$  such that \begin{enumerate} \item[(*)]  $  K ~\cup~{\bigcup_{n=1}^{\infty}p_n^{-1}(0)}\subset
\D(0,\delta)$, \item[(**)]
  $ \forall z\in \cj\setminus \D(0,\delta) \ \exists U$ -- a neighbourhood of $z$ such that:
\begin{equation*}\label{e:ag-b}
\exists N=N(U)\geq 1\ \forall n\geq N \ \forall \zeta\in U:\;   \frac{1}{n}\log^+|p_n(\zeta)| \geq ag_K(\zeta)-b.
\end{equation*}
 \end{enumerate}
\end{definition}
  Actually, in \cite[Definition 3.1]{CHPP}  $K$ belonged to a bigger family of compact sets, where   a more general definition of the Green function was applicable.
 The authors  also assumed instead of (**) that inequality $$\liminf_{n\rightarrow \infty} \frac{1}{n}\log^+|p_n(z)| \geq ag_K(z)-b$$ should hold ``locally uniformly'' in  $\mathbb{C}\setminus {\rm conv} K$ (instead of in $\mathbb{C}\setminus \overline{\mathbb{D}}(0,\delta)$). We prefer to write the explicit meaning of this notion, namely (**), following personal  communication with the authors.

 Observe  that because of the non-negativity of $g_K$, if $(p_n)_{n=1}^\infty$ is $K$-guided with a number $a$,   it is also $K$-guided with any number $a'\in (0,a)$. In particular without loss of generality one can assume that $a\in (0,1]$.
  On the other hand,   $a \leq 1$  necessarily  for  sequences satisfying an additional assumption, see Proposition \ref{prop: jedynka}.

Note that the assumption  $\deg p_n=n$ can be relaxed (in particular, all polynomials may be allowed to have the same degree).   The fractions $\frac{1}{n}$ in (**) and in all statements based on this condition should be then replaced by $\frac{1}{\deg p_n}$, $n\in\{1,2,...\}$.
 However,   in what follows  only  finitely many polynomials   will be    allowed to be of degree 1, the others must have higher degrees. Equivalently, $\alpha:=\limsup_{n \to \infty} \frac{1}{{\rm deg} \ p_n} \leq 1/2$.   In this case,  without loss of generality we can assume that only $p_1$ is allowed to be of degree 1. Indeed, if $\nu_0:=\min\{\nu \in \mathbb{N}|\ \forall n \geq \nu: \ \deg p_n \geq 2\}\geq 3$, we can work instead with polynomials $P_1:=p_{\nu_0-1}\circ...\circ p_1, P_2:=p_{\nu_0},..., P_n:=p_{\nu_0-2+n}, \ n \geq 2$.  Another possibility is to take $P_1$ to be the composition of all polynomials of degree 1 in the sequence   (if any)  and then put $P_2:=p_{k_2}$ with $k_2=\min\{n: \deg p_n>1\}$,  $P_3:=p_{k_3}$ with $k_3=\min\{n\neq k_2: \deg p_n>1\}$ and so on.

We will now prove an important technical result. Note that a $K$-guided sequence (in the more general sense as explained above) satisfies the assumptions of the following Lemma \ref{prop: tnzez} provided $a>\alpha$ (with a small change of $b$ if needed).

\begin{lemma} \label{prop: tnzez} Let $(p_n)_{n=1}^\infty$ be a  sequence of polynomials with   $\deg p_1 \geq 1 $ and $\deg p_n \geq 2 $ for $n \in \{2,3,...\}$.
  Define $\alpha:=\limsup_{n \to \infty} \frac{1}{\deg p_n}$.
   Assume  that there exist an $r>0$ and a function $h:\cj\longrightarrow \rj$ satisfying $h(z)-\alpha\log|z|\longrightarrow \infty \text{ as } |z|\to \infty$  such that   \begin{enumerate}
\item[$(i)$] $
\bigcup_{n=1}^\infty p_n^{-1}(0) \subset
 \D(0,r)$;
 \item[$(ii)$] $\forall z\in \mathbb{C} \setminus \D(0,r) \ \exists U$ -- a neighbourhood of $z$ such that
 \[\exists N=N(U)>1\ \forall n\geq N\ \forall \zeta\in U:\; \frac1{\deg p_n}\log^+|p_n(\zeta)|>h(\zeta).\]
                                                                                                          \end{enumerate}

Then
$\ \exists R>0\ \forall n\geq 2:\quad |z|\geq R\quad \Longrightarrow \quad |p_n(z)|\geq e|z|.$
\end{lemma}

\begin{proof}
 Choose $R_0 > \max\{2,2r\}$ such that
 \begin{equation}\label{e:h}h(z) -\alpha \log|z|>2 \quad \text{ if }\ |z| \geq R_0/2.\end{equation}

 Note first that  since $\alpha=\limsup_{n\rightarrow \infty} \frac1{{\rm deg} \ p_n}$,
 \begin{equation}\label{e:alpha}\exists N^\star>1\ \forall n\geq N^\star:\quad
 \frac1{{\rm deg} \ p_n}\log R_0\leq \alpha\log R_0+1.
 \end{equation}

By $(ii)$ if $z\in \partial \D(0,R_0)$, then there exist a neighbourhood  $U$ of $z$ and an integer $N_z$ such that
$\forall \zeta\in U\ \forall n\geq N_z:\; \frac1{{\rm deg} \ p_n}\log^+|p_n(\zeta)|>h(\zeta).$
By compactness of the circle there exists $N\geq N^\star$ such that
\begin{equation*}
  \forall \zeta\in \partial\D(0,R_0)\ \forall n\geq N:\quad  \frac1{{\rm deg} \ p_n}\log^+|p_n(\zeta)|>h(\zeta).
\end{equation*}
Note that because of (\ref{e:h}) the right hand side of the inequality above is greater than 2, hence we can   replace  $\log^+$ with $\log$, obtaining
\begin{equation}\label{e:p_n}
  \forall \zeta\in \partial\D(0,R_0)\ \forall n\geq N:\quad  \frac1{{\rm deg} \ p_n}\log |p_n(\zeta)|>h(\zeta).
\end{equation}
  Recall  that $\alpha \in [0,1/2]$.
Combining (\ref{e:h}), (\ref{e:alpha}) and (\ref{e:p_n}), we get
\begin{align*}
  \forall \zeta\in \partial\D(0,R_0)\ \forall n\geq N: \quad
\frac1{{\rm deg} \ p_n}\log |p_n(\zeta)|-\frac1{{\rm deg} \ p_n}\log |\zeta|>\\ h(\zeta)-\alpha\log |\zeta|-1> 1.
\end{align*}
 In consequence,  for all $\zeta\in\partial \D(0,R_0)$
\begin{equation}\label{e:varepsilon}
  \forall n\geq N:\quad \log |p_n(\zeta)|- \log |\zeta| > 1.
\end{equation}

Since all zeros of the  polynomials $p_n$ are contained in
$\D(0,r) \varsubsetneq \D(0,R_0)$,  we may apply the Minimum Principle for harmonic functions to  $z\longmapsto  \log|p_n(z)|-  \log|z| $
for $ n \geq N $ and conclude that    (\ref{e:varepsilon}) is satisfied for all $\zeta $ in the complement of the open disk $\mathbb{D}(0,R_0)$. Thus
$$
\forall n\geq N:\quad  |z|\geq R_0\quad \Longrightarrow \quad |p_n(z)| \geq  e|z|.
$$

Now, by Lemma \ref{lem: theta}
$$\forall j\in\{2,3,...,N-1\}\ \exists R_j>0:\quad |z|\geq R_j\quad \Longrightarrow \quad |p_j(z)|\geq e|z|.$$
It suffices to take $R:=\max\{R_j: j\in\{0,2,3,...,N-1\}\}.$
\end{proof}

\begin{cor}\label{cor:przeciwobrazkola}
   If $(p_n)_{n=1}^\infty$ is a   sequence of   polynomials satisfying the assumptions of Lemma  $\ref{prop: tnzez}$,
   then
  $$\exists \varrho>0\; \forall R\geq \varrho\; \forall n\geq 2:\quad  p_n^{-1}\left(\D(0,R)\right)\subset \D(0,R).$$
  \end{cor}

Let  us propose another definition.

\begin{definition} \label{def:guided} Consider a sequence of polynomials $(p_n)_{n=1}^\infty$ such that   $\deg p_1 \geq 1 $ and
$\deg p_n \geq 2 $ for $n \in \{2,3,...\}$.  Define $\alpha:=\limsup_{n \to \infty} \frac{1}{\deg p_n}$.    We say that $(p_n)_{n=1}^\infty$ is a {\it guided} sequence of polynomials if  there exist an $r>0$ and a function $h:\cj\longrightarrow \rj$ satisfying $h(z)-\alpha\log|z|\longrightarrow \infty \text{ as } |z|\to \infty$  such that    \begin{enumerate}
\item[$(i)$] $
\bigcup_{n=1}^\infty p_n^{-1}(0) \subset
 \D(0,r)$;
 \item[$(ii)$]
$\exists N >1\ \forall n\geq N\ \forall \zeta\in  \mathbb{C} \setminus \D(0,r):\;\ \frac1{\deg p_n}\log^+|p_n(\zeta)|>h(\zeta).$
                                                                                                          \end{enumerate}
\end{definition}

 Observe that if $(p_n)_{n=1}^\infty$ is a guided sequence, then for every mapping $\varphi:\nj\longrightarrow \nj$ such that 
   $\varphi(k)>1$ for all $k>1$ , the sequence $(p_{\varphi(n)})_{n=1}^\infty$ is also  guided. If $\deg p_1>1$, the assumption 
 $\varphi(k)>1$ for all $k>1$  can be omitted.

 Note that every guided sequence $(p_n)_{n=1}^\infty$   satisfies the assumptions of Lemma \ref{prop: tnzez}. Moreover, $$ \forall z\in\cj\setminus \D(0,r):\quad \liminf_{n \to \infty} \frac{1}{\deg p_n}\log^+|p_n(z)|\geq h(z).$$

\begin{example}\label{ex:guidedinfty}
If $p_n(z):=n^nz^n$ for $n \geq 1$, then $(p_n)_{n=1}^\infty$ is guided, since   $\alpha =0$ and
 $$\forall n\geq 2 \ \forall z\in\cj\setminus \D(0,1): \quad \frac{1}{n}\log^+|p_n(z)|=\log n+\log^+|z|> \log^+|z|.$$
\end{example}

There are of course sequences that are not guided. First one can take polynomials with zeros that are not uniformly bounded, e.g. $p_n(z)=z^n-n^n$, $n\geq 1$. Note also the following example.

\begin{example}
Let $p_n(z)=2^{-n^2}z^n$ for $ n\geq 1$. Then for any fixed $N\geq 1$ we have $   |z|\leq 2^{N} \Longrightarrow \forall n\geq N:\ |z^n|\leq 2^{n^2}$. Fix $z\in \cj$. We have $\liminf_{n\rightarrow +\infty} \frac1n\log^+|p_n(z)|=0$, since  $\frac1n\log^+|p_n(z)|\neq 0$ only for a finite number of integers $n$. Note that  the zeros of $p_n$ are uniformly bounded.
\end{example}

 We are ready to present the
main result of this section -- a characterization of  guided sequences.
 We relax the assumptions on the degrees for the $K$-guided sequences as suggested under Definition \ref{def:K-guided}.

\begin{theorem} \label{thm: guided} Let $(p_n)_{n=1}^\infty$ be a sequence of complex polynomials such that  $\deg p_1 \geq 1$ and $\deg p_n \geq 2$ for $n \in \{2,3,...\}$.  Define
 $\alpha:=\limsup_{n \to \infty} \frac{1}{\deg p_n}$.

The following conditions are equivalent:
\begin{enumerate}
\item[$(i)$] $(p_n)_{n=1}^\infty$ is a guided sequence;

 \item[$(ii)$] $(p_n)_{n=1}^\infty$ satisfies the assumptions of Lemma $\ref{prop: tnzez}$;

\item[$(iii)$] $\exists R>1\ \forall n   \geq 2: \quad  p_n^{-1}\left(\D(0,R)\right)\subset \D(0,R)$;

 \item[$(iv)$]
 For every $K\subset \cj$ compact  and regular  the sequence $(p_n)_{n=1}^\infty$ is $K$-guided  with $a\in(\alpha,1]$;

\item[$(v)$] There exists $K\subset \cj$ compact    and regular  such that $(p_n)_{n=1}^\infty$ is $K$-guided   with $a\in(\alpha,1]$.\end{enumerate}
\end{theorem}

\begin{proof}   Note first that if $(p_n)_{n=1}^\infty$ satisfies any of conditions $(i)$-$(v)$, then the zeros of the polynomials from the sequence are uniformly bounded.

$(i)\Longrightarrow (ii)$ and $(iv) \Longrightarrow (v)$ are obvious.

$(v) \Longrightarrow (ii)$ It suffices to take $h:=ag_K-b-1$.
The condition $\lim_{|z| \to \infty}(h(z)-\alpha \log|z|)=+\infty$ holds, since $g_K \in \mathcal{L}^+$ in view of (\ref{e:L+})  and $a>\alpha$.

$(ii) \Longrightarrow (iii)$    This is    Corollary \ref{cor:przeciwobrazkola}.

$(iii) \Longrightarrow (iv)$ and $(iii)\Longrightarrow (i)$ Fix a   compact   regular  set $K\subset \cj$. Let $R>1$ be as in $(iii)$ and let $\varepsilon >0$ be such that $\overline{\mathbb{D}}(0,R)
\subset K_\varepsilon$. We have $\bigcup_{ n \geq 2}p_n^{-1}\left(\overline{\mathbb{D}}(0,1)\right) \subset \bigcup_{ n \geq 2}p_n^{-1}\left(\D(0,R)\right)
\subset K_\varepsilon$.

 It follows from
 (\ref{e:AsubsetB}),  (\ref{e:Greenaprzeciwobrazu}) and Example \ref{ex: koloGreen} that  $$\forall z\in\cj\ \forall  n\geq 2:\quad g_{K_\varepsilon}(z) \leq g_{p_n^{-1}\left(\D(0,1)\right)}(z) = \frac{1}{\deg p_n} \log^+|p_n(z)|.$$   To get $(iv)$ with $K$ and $a\in(\alpha,1]$ it suffices now to recall that $g_{K_\varepsilon}  = g_K  -\varepsilon$ in $
 \mathbb{C}\setminus K_\varepsilon$ by (\ref{e:Mazurek}) and that $g_K$ is non-negative.  In view of (\ref{e:L+}),  $(i)$ is also satisfied (e.g. with $h=g_{K_\varepsilon}-1$),
 since $\alpha\in[0,\frac12]$.
\end{proof}

 \begin{remark} A property slightly weaker than  Lemma \ref{prop: tnzez} and the implication $(iv) \Longrightarrow   (iii)$ in Theorem \ref{thm: guided} were proved under different assumptions as  \cite[Proposition 3.3]{CHPP}.
\end{remark}

 \subsection{Examples of guided sequences of polynomials} \label{subsection: KW}

   The notion of guided sequences generalizes the following classical one.

\begin{definition}\label{def:KWpolynomials}
  Let  $K \subset \mathbb{C}$ be a regular polynomially convex compact set. Consider a sequence of polynomials $(p_n)_{n=1}^\infty$ such that $$p_n:\cj\ni z \longmapsto \big(z-\zeta_1^{(n)}\big)...\big(z-\zeta_n^{(n)}\big)\in\cj ,\quad n\in\{1,2,...\}.$$
  (Thus all $p_n$ are monic 
  and $\forall n: \  \deg p_n=n.$)
  We say that $(p_n)_{n=1}^\infty$  satisfies  the {\it Kalm\'ar-Walsh condition}, or is a {\it KW sequence} (of polynomials)
   associated with $K$, if  the set $\bigcup_{n \in \mathbb{N}}\left\{\zeta_1^{(n)},...,\zeta_n^{(n)}\right\}$ is bounded and
   \begin{equation}\label{eq: normy}
 \lim\limits_{n\to\infty} \|p_n\|_K^{1/n} = {\rm cap}(K).
 \end{equation}
\end{definition}

When zeros of all $p_n$ do not have accumulation points in $\mathbb{C}\setminus K$, this is equivalent to  $\lim\limits_{n\to\infty} |p_n(z)|^{1/n} = {\rm cap}(K)\cdot \exp(g_K(z))$ uniformly on compact subsets of $\cj\setminus K$
(cf.   \cite[Theorem 1.4]{BBCL}, \cite[Section 7.3, Theorem 3 and Section 7.4, Theorem 4]{Walsh} and  \cite[II.2.B, Theorem 1 and Lemma 1]{Gaier}).   In particular, in this setting a KW sequence of polynomials associated with $K$ is $K$-guided, with $a=1$ and  e.g. $b=-\log{\rm cap}(K)-1$.  For  the case when the zeros of $p_n$ are assumed to be uniformly bounded but are allowed to have limit points in the unbounded component of $\mathbb{C}\setminus K$   see \cite{Leja47}  (cf. also the proof of Proposition \ref{prop:KWfinite}).

Let us also recall the known notion of Chebyshev polynomials.

\begin{definition}\label{def:nthCheb} Let $E \subset \mathbb{C}$ be a compact set. A monic polynomial $t_n$  of degree $n \geq 1$
 is called the  {\it $n$th Chebyshev polynomial} (or the {\it $n$th minimal polynomial}) on $E$ if $\|t_n\|_E \leq \|q\|_E$ for any monic polynomial $q$ of degree $n$.
\end{definition}

  An important example is associated with the interval $[-1,1]$ (cf. Main Example from the Introduction).
\begin{example}\label{ex:Td}
  Let $E=[-1,1]$  and $n\geq 1$. Then $t_n:=\frac{1}{2^{n-1}}T_n$ is the $n$th
  minimal
polynomial on $E$, where   $T_n$ satisfies the  formula $T_n\left(\frac{z+ z^{-1}}{2}\right)=\frac{z^n+z^{-n}}{2}$. The polynomials $T_n$ are called the {\it  classical Chebyshev polynomials}.
\end{example}

 We observe  the following fact:

\begin{proposition}\label{prop:czebyszewKW}
  If $K\subset \cj$ is compact, polynomially convex and regular, then the sequence of Chebyshev polynomials on $K$ is a KW sequence associated with $K$.  \end{proposition}
\begin{proof}
 Fej\'er showed in \cite{Fejer22} that all zeros of the Chebyshev polynomials on a compact set $K$ lie in the convex hull   ${\rm conv}K$ (see also e.g. \cite{Leja47}).
Moreover, in \cite{Fekete} it was proved that (\ref{eq: normy})  holds for the sequence of minimal polynomials  on a compact set.
\end{proof}

More examples of  sequences $(p_n)_{n=1}^\infty$   such that $\bigcup_{n \in \mathbb{N}}\left\{\zeta_1^{(n)},...,\zeta_n^{(n)}\right\}\subset K$  and   (\ref{eq: normy}) is satisfied can be found in \cite[II.2]{Gaier}.

Now we turn to other
examples. The authors of \cite{BB03}
  introduced a class $\mathcal{B}$
 of sequences $(f_n)_{n=1}^\infty$, where $f_n(z) =
\sum_{j=1}^{d_n}
 a_{n,j}z^j$ and $d_n \geq 2$ for
any $n$. Formally, in \cite[4.1]{BB03}  $\mathcal{B}$ was defined as the class of sequences of polynomials satisfying a variant  of the condition $(iii)$ in our Theorem \ref{thm: guided}. Therefore such sequences are guided. However, in practice, in \cite{BB03} as well as in other works building on it, only sequences of polynomials satisfying some  sufficient condition were considered to be members of $\mathcal{B}$.
In particular, they would satisfy
\begin{enumerate}
  \item[(P2)] there is a constant $A\geq 0$ such that $|a_{n,j}|\leq A|a_{n,d_n}|$ for $j\in\{0,...,d_n\}$ and all integers $n$.
\end{enumerate}
The  necessity of
condition (P2) for membership in $\mathcal{B}$ was not addressed in the existing literature. We would like to point out that (P2) does not hold for certain KW sequences of polynomials.   Therefore the class of guided sequences is larger than the class of sequences satisfying (P2). Here is an explicit example.


\begin{remark}\label{rem:notinB}  The sequence of the Chebyshev polynomials on $[-1, 1]$  does not
  satisfy (P2). 
\end{remark}

\begin{proof}
 Recall that $T_n$ is the classical Chebyshev polynomial of degree $n$ and the $n$th Chebyshev polynomial on $[-1,1]$ is $\frac{1}{2^{n-1}}T_n$ (see Example \ref{ex:Td}). We may write  $T_n(z)=2^{n-1}z^n+a_{n-1}^{(n)}z^{n-1}+...+a_0^{(n)}$. Note that $a_{n-1}^{(n)}=0$ for every $n \geq 1$.

First we will prove that $a_{n-2}^{(n)}=-n2^{n-3}$ for $ n\in\{1,2,...\}$. Indeed, this is true for $T_1(z)=z, \ T_2(z)=2z^2-1, \ T_3(z)=4z^3-3z$. To argue by induction, let $n$ be such that $a_{k-2}^{(k)}=-k2^{k-3}$ for $k\in\{2,3,...,n\}$. From the recurrence formula $T_{n+1}(z)=2zT_n(z)-T_{n-1}(z)$ (valid for all $n \geq 1$, if we set $T_0(z)\equiv 1$) we can express $T_{n+1}(z)$ as
\begin{align*}
2z\left(2^{n-1}z^n+(-n2^{n-3})z^{n-2}+a_{n-3}^{(n)}z^{n-3}+...+a_0^{(n)}\right)+\\-\left(2^{n-2}z^{n-1}+(-(n-1))2^{n-4}z^{n-3}+a_{n-4}^{(n-1)}z^{n-4}+...+a_0^{(n-1)}\right).
\end{align*}
Then $a_{n-1}^{(n+1)}=-n2^{n-2}-2^{n-2}=-(n+1)2^{n-2}$, as claimed.

  Thus for the  polynomial $\frac{1}{2^{n-1}}T_n$ the coefficient corresponding to $z^{n-2}$ is  $-n/4$. Hence    condition (P2) does not hold.
\end{proof}

For the explicit expression of coefficients of $T_n$ corresponding to the powers of the variable $z$  see  e.g. \cite[T1.2 (40)]{Pasz75}.

 Assume now that $(p_n)_{n=1}^\infty$  is guided.   Observe that it need not satisfy the Kalm\'ar-Walsh condition, since  the degrees of polynomials do not have to tend   to infinity.  One could ask whether   the assumption   $\deg p_n \longrightarrow \infty$ as $n \to \infty$ implies that  $(p_n)_{n=1}^\infty$ must be a (subsequence of a)  KW sequence. In the following example, given any compact regular polynomially convex set $K \subset \mathbb{C}$, we will show that there is a guided sequence of polynomials   with  $\deg p_n \longrightarrow \infty$ as $n \to \infty$  which
does not satisfy the Kalm\'ar-Walsh condition for  $K$.

\begin{example} \label{ex: nonKW}   Fix    a compact regular polynomially convex  set $K\subset \mathbb{C}$ and $a \in (0,1)$. We will construct a $K$-guided sequence $(P_n)_{n=1}^\infty $ of polynomials  satisfying estimates (**) from Definition  \ref{def:K-guided} with the chosen $a$ and with  $b=1$ 
 which is not   (a subsequence of) a KW sequence for $K$.

  Let  $L\subset \mathbb{C}$ be any   compact regular polynomially convex set such that $K \cap L=\emptyset$ and $\mathbb{C}\setminus (K \cup L)$ is connected.  Take  $\delta>0$ such that $K \cup L\subset \D(0,\delta)$.   Let $(p_n)_{n=1}^\infty$ and $(q_n)_{n=1}^\infty$ be
 KW sequences of polynomials associated with $K$ and $L$ respectively,
such that all zeros of $p_n$ lie in $K$ and all zeros of $q_n$ lie in $L$.
Let further $k_n,l_n$ be two sequences of integers such that  $0<k_n<l_n$, $n\in\{1,2,...\}$, and $a=\lim_{n_\to \infty}k_n/l_n$.  If $a=k/l\in \mathbb{Q}$ with $k,l$ coprime, simply take the sequences $k_n, l_n$ to be constant, $k_n=k$, $l_n=l$.  If $a \notin \mathbb{Q}$,  take $l_n \longrightarrow \infty$ as $n \to \infty$ (e.g., fixing an approximation of $a$ by decimal fractions).

Let $P_n:=(p_n)^{k_n}(q_n)^{l_n-k_n}$ for $n\in\{1,2,...\}$. Then ${\rm deg} P_n=nl_n$. Furthermore,
\[
\lim_{n \to \infty}\frac{1}{nl_n}\log|P_n(z)|=\lim_{n \to \infty}\frac{k_n}{l_n}\frac{1}{n}\log|p_n(z)|+\lim_{n \to \infty}\frac{l_n-k_n}{l_n}\frac{1}{n}\log|q_n(z)|,
\]
which equals $ag_K(z)+(1-a)g_L(z)$ in  $\cj\setminus \D(0, \delta)$
and the convergence is locally uniform. It follows that
 $  \forall z\in \mathbb{C}\setminus \D(0, \delta) \ \exists U$ -- a neighbourhood of $z$ such that
 \[ \exists N\geq 1\ \forall n\geq N\ \forall \zeta\in U: \quad \frac1{\deg P_n} \log^+|P_n(\zeta)| \geq ag_K(\zeta)-1.
\]
That is, $(P_n)_{n=1}^\infty$ is a $K$-guided sequence of polynomials with the chosen value of $a$  and with 
 $b=1$.

 Suppose $ag_K(z)+(1-a)g_L(z)=g_K(z)$ when $z\in  \mathbb{C}\setminus \D(0, \delta)$.   By the identity principle for harmonic functions, $ag_K(z)+(1-a)g_L(z)=g_K(z)$ when $z\in\mathbb{C}\setminus(K \cup L)=:\Omega$. It follows that  $g_K\equiv g_L$ in $\Omega$.
 Thus the continuity of the Green functions and the connectedness of $\Omega$
 imply that
 $\forall z\in \partial L: \ g_K(z)=0$. This is a contradiction because $K\cap L=\emptyset$.

Note that this argument also shows that  $(P_n)_{n=1}^\infty $ is not  (a subsequence of) a KW sequence for $L$.
\end{example}


\section{Polynomial Julia type sets} \label{s: Julia}

\subsection{Autonomous Julia sets}

\begin{definition}\label{def:filledJulia}
  Let $P:\cj\longrightarrow \cj
  $ be a polynomial of degree $d\geq 2$. The  {\it (autonomous) filled Julia set} of $P$ is
\begin{align*}
\K[P]:=&\ \{z\in\cj:\; (P^n(z))_{n=1}^\infty \text{ is bounded}\}.
\end{align*}
\end{definition}
It is well known that the filled Julia set is non-empty, compact,  totally invariant under $P$ (i.e., $P(\K[P])=P^{-1}(\K[P])=\K[P]$) and
\begin{align}\label{e:infty}
\K[P]:=  \cj\setminus \left\{z\in\cj:\:  \lim_{n\rightarrow \infty} P^n(z)= \infty\right\}.
\end{align}
  For more information, see e.g. \cite{CGbook}.
\begin{example}\label{ex: filled}
$ \forall n\geq 2:\; \K[z\longmapsto z^n]=\D(0,1)\; \text{ and } \; \K[T_n]=[-1,1],$
where $T_n$ is the $n$th classical Chebyshev polynomial for $n\geq 2$.
\end{example}

Let us define the notion of   an  escape radius of a polynomial, following \cite[page 53]{KK18} (a slightly different definition was proposed in \cite{Douady}).

\begin{definition}\label{def:escape} An {\it escape radius} for a polynomial $P:\cj\longrightarrow \cj$   is a number $R>0$ with the following property
$$|z|>R\quad \Longrightarrow \quad  \lim_{n\rightarrow \infty} P^n(z)= \infty.$$
\end{definition}

\begin{lemma}\label{lem:escaperadius}
  $R>0$ is an escape radius for a polynomial $P$ if and only if $
                                                         \K[P]\subset \D(0,R).
                                                      $
In particular, if $R$ is an escape radius for $P$ and $\varrho>R$, then $\varrho$ is an escape radius for $P$ too.
\end{lemma}
\begin{proof}
  This follows from the definition and from (\ref{e:infty}).
\end{proof}

\begin{cor}\label{cor: intersecting} If $R>0$ is an escape radius for a polynomial $P$, then
\[
\K[P]=\bigcap_{n=1}^{\infty}P^{-n}\left(\D(0,R)\right).
\]   \end{cor}
\begin{proof} Let $R>0$ be such that $\K[P]\subset \D(0,R)$.
In consequence  $\mathcal{K}[P]=P^{-n}(\mathcal{K}[P]) \subset P^{-n}(\overline{\mathbb{D}}(0,R))$ for each $n \geq 1$.
 Conversely, if $\forall n\geq 1:\ z \in P^{-n}\left(\D(0,R)\right)$, then $(P^n(z))_{n=1}^{\infty} \subset \D(0,R)$, hence $z \in \K[P]$.
\end{proof}

A sufficient condition for a positive number to be an escape radius for a polynomial  was given in Lemma \ref{lem: theta}.

\subsection{ Non-autonomous Julia sets}

We will consider now a generalization of the autonomous 
 Julia set.   Even though usually only polynomials of degree greater than 1 are considered (cf. e.g. \cite{B97}, \cite{KK18}), we will allow one polynomial to be of degree 1. This in particular takes into account  KW sequences
of polynomials and enables their study.

\begin{definition}\label{def:nonautonomous}
   Let $(p_n)_{n=1}^\infty$ be a sequence of polynomials such that $\deg p_1\geq 1$ and $\deg p_n\geq 2$ for $n\geq 2$.
  We define the {\it (non-autonomous) filled Julia set} of the sequence $(p_n)_{n=1}^\infty$ to be
  \begin{align*}
\K[(p_n)_{n=1}^\infty]:=&\ \left\{z\in\cj:\; \left((p_n\circ\dots\circ p_1)(z)\right)_{n=1}^\infty \text{ is bounded}\right\}.
\end{align*}
\end{definition}

\begin{remark}\label{rem:pierwszy}
  It is straightforward that in the situation from the definition $$\K[(p_n)_{n=1}^\infty]=p_1^{-1}\left(\K[(p_n)_{n=2}^\infty]\right).$$
  In particular if $p_1:\cj\ni z \longmapsto z\in\cj$, then
$\K[(p_n)_{n=1}^\infty]=\K[(p_n)_{n=2}^\infty]$.
\end{remark}

Note that it follows from   Definition \ref{def:nonautonomous} that
$$\K[(p_n)_{n=1}^\infty]=\bigcup_{r\in \nj}\bigcap_{n\geq 1}(p_n\circ...\circ p_1)^{-1}\left(\D(0,r)\right),$$
hence  $\K[(p_n)_{n=1}^\infty]$ is of $F_\sigma$-type.

 A non-autonomous filled Julia set may be finite, which is impossible for an autonomous one.
  \begin{example} \label{ex:juliafinite}
 As shown in \cite{B97}, if we take $\forall n\geq 1:\ p_n:\cj\ni z \longmapsto n^{2^n}z^2\in\cj$, then $\K[(p_n)_{n=1}^\infty]=\{0\}$. Similarly, if we
  take $q_1:\cj\ni z\longmapsto z^2-1\in\cj$ and $\forall n\geq 2: q_n:=p_n$, then $\K[(q_n)_{n=1}^\infty]=\{-1,1\}$. Both  sequences are guided.\end{example}

If $(p_n)_{n=1}^\infty$ in Definition \ref{def:nonautonomous} above is periodic (i.e., there exists an $m$ such that $p_{m+i}=p_i$ for every $i$),  then we obtain the autonomous filled Julia set (see Definition \ref{def:filledJulia})
of the polynomial $p_m\circ....\circ p_1$.
In particular for a constant sequence we also get   the  
autonomous Julia set. The interesting case is when the sequence is not periodic, e.g., when each polynomial in the sequence has a different degree.

\begin{example}\label{ex:circ} If $(d_n)_{n=1}^\infty$ is a sequence of integers not smaller than 2 and $p_n:\cj\ni z \longmapsto z^{d_n}$, then $\K[(p_n)_{n=1}^\infty]=\D(0,1)$. Furthermore $\K[(T_{d_n})_{n=1}^\infty]=[-1,1]$. This follows from Example \ref{ex: filled} and the  composition formulae $p_n\circ p_k=p_{nk}$ and $T_n\circ T_k=T_{nk}$.  In particular $\K[(z\longmapsto z^n)_{n=1}^\infty]=\D(0,1)$ and $\K[(T_n)_{n=1}^\infty]=[-1,1]$.
\end{example}

Let us note a consequence of
Example \ref{ex:circ} and Remark \ref{rem:pierwszy}.

\begin{example}
  For any non-constant polynomial $P:\cj\longrightarrow \cj$, the sets $P^{-1}\left(\D(0,1\right))$ and $P^{-1}\left([-1,1]\right)$ are (non-autonomous) filled Julia sets. \end{example}

The following theorem is our first result about the non-autonomous Julia set defined with  the  use of a  guided sequence of polynomials.

\begin{theorem}\label{thm:JuliadlaKW}
   Let  $(p_n)_{n=1}^\infty$ be a  guided sequence of  polynomials.
   Then $ \K[(p_{n})_{n=1}^\infty]$ is   non-empty and  compact.
\end{theorem}

\begin{proof}
  Let $R$ be as in Lemma \ref{prop: tnzez}. We have  $$|z|\geq R\quad \Longrightarrow \quad  |(p_{n}\circ...\circ p_{2})(z)|\geq e^{n-1}|z|\longrightarrow \infty, \text{ if } n\rightarrow \infty.$$  Therefore $\K[(p_{n})_{n=2}^\infty]=\bigcap_{n=2}^{\infty}(p_{n}\circ...\circ p_{2})^{-1}\left( \D(0,R)\right)$.
It follows that $ \K[(p_{n})_{n=1}^\infty]$ is non-empty and compact by Remark \ref{rem:pierwszy}.
\end{proof}

  Recall that the  Julia set  obtained in this theorem may be finite, as shown in Example \ref{ex:juliafinite}. Our aim is to find regular sets,  therefore we need an additional assumption.
Remark \ref{rem:notinB} shows that the approach from \cite{BB03}, where bounds on coefficients were assumed,    may be insufficient  in our study.
Another   course of action
was proposed in \cite{KK18},
where it was required  that all polynomials from the sequence $(p_n)_{n=1}^\infty$ have  a common escape radius $R$ such that
\begin{equation}\label{e:suppn}\sup_{n}\|p_n\|_{\D(0,R)}<\infty.\end{equation}
In this case $\K[(p_n)_{n=1}^\infty]$ is  regular.
By Lemma  \ref{prop: tnzez}, polynomials in a guided sequence have a common escape radius.  But  the condition (\ref{e:suppn}) does not have to be satisfied,  which can be seen in the following example.

 \begin{example}\label{ex:bounded} Consider  $p_1:\cj\ni z \longmapsto z^2-2\in\cj$ and $p_n:\cj\ni z\longmapsto z^n\in\cj$ for $n\geq 2$. The sequence $(p_n)_{n=1}^\infty$ is obviously guided. Note that
$\K[p_1]=[-2,2]$, therefore in view of Lemma \ref{lem:escaperadius} the smallest escape radius for $p_1$ is 2. Therefore  $R $ is a common escape radius for   $(p_n)_{n=1}^\infty$ if and only if $R\geq 2$. However, $\sup_{n\geq 1}\|p_n\|_{\D(0,2)}\geq \sup_n2^n=+\infty$. On the other hand
$$\sup_{n\geq 1}\frac1{\deg p_n}\log^+|p_n(z)|= \max\left\{\frac12\log^+|z^2-2|, \log^+|z|\right\}<+\infty$$ on any bounded subset of $\cj$.\end{example}

We are interested in  a less restrictive sufficient  condition
for the non-autonomous Julia set defined by a guided sequence of polynomials $(p_n)_{n=1}^\infty$ to be
regular.
 It turns out that the right condition to impose~~is
   \begin{equation}\label{e:finite}  \sup_{n \geq 1} \frac{1}{\deg  p_n} \log^+|p_n(z)|<+\infty \text{ on an open set } D \subset \mathbb{C}.  \end{equation}
 Obviously (\ref{e:suppn}) implies (\ref{e:finite})  and Example \ref{ex:bounded} shows
  that the opposite implication is not true. Therefore the approach from \cite{KK18} is also insufficient for our study.  However, we will use  some tools and results appearing there.

  Before we go further let us comment more on condition (\ref{e:finite}).
 Recall that a guided sequence does not need to satisfy it,
 as shown in Example \ref{ex:guidedinfty}. On the other hand we have the following  results.
 The first one confirms that the approach from \cite{BB03} is insufficient for our case.

  \begin{proposition}\label{prop:KWfinite}
    Every KW sequence of polynomials satisfies $(\ref{e:finite})$.
  \end{proposition}
\begin{proof}
Assume that $K\subset \cj$ is compact, regular and polynomially convex and $(p_n)_{n=1}^\infty$ is a KW sequence
(see Definition \ref{def:KWpolynomials}) associated with $K$ with all zeros in $\D(0,r)$ for some $r>0$. We recall a result from \cite{Leja47}. It was  stated and proved  there for Chebyshev (minimal) polynomials on a compact set $K$, but it can be easily generalized to our setting. Following the lines of the proof of that result, one can  namely show 
that (\ref{eq: normy}) implies locally uniform convergence on $\cj\setminus \D(0,r)$ of the sequence $(\frac1n\log|p_n|)_{n=1}^\infty$ to $g_K+\log {\rm cap}(K)$. This  combined with the continuity of $g_K$ implies (\ref{e:finite}) for a suitable open set outside of $\D(0,r)$.
\end{proof}

 The second result is the promised relation for $K$-guided sequences.

\begin{proposition}\label{prop: jedynka}
Let $K\subset \cj$ be compact and regular and $(p_n)_{n=1}^\infty$ be  $K$-guided  with $a>0$ and $b\in\rj$. If $(p_n)_{n=1}^\infty$ satisfies $(\ref{e:finite})$,  then $a\leq 1$.
\end{proposition}

\begin{proof}
 Define $u:=\sup_{n \geq 1}\frac{1}{{\rm deg} \ p_n}\log^+|p_n|$ and take $\delta$  from Definition \ref{def:K-guided}.  In view of (\ref{e:finite})
 Proposition \ref{prop: u*inL} yields  $u^* \in \mathcal{L}$.
Recall that $g_K \in \mathcal{L}^+$   by (\ref{e:L+}). Thus there exist $R>\delta$, \ $C_1>0, \  C_2 >0$ such that
\[
u^*(z) < C_1+\log|z| \quad \text{ and } \quad C_2 + a\log|z| < ag_K(z)
\]
when $|z|>R$. But $u \leq u^*$ and $(p_n)_{n=1}^\infty$ is $K$-guided, hence
\[
|z|>R \quad \Longrightarrow \quad C_2 + a\log|z| -b  < ag_K(z)-b <  C_1 +\log|z|.
\]

If $a> 1$, take $R':=\max\left\{ R, \exp \frac{C_1+b-C_2}{a-1}\right\}+1$ and let $|z|>R'$.  Then $(a-1)\log|z| < C_1+b-C_2$, which implies   $|z|<\exp \frac{C_1+b-C_2}{a-1}< R'$,  a contradiction.
\end{proof}


\subsection{ Klimek's metric}\label{s:Klimek}

We use the following notation
\begin{align*}
  \R&={\R}(\cj)=\\&:=
  \{K\subset \cj:\;  K \text{ is   compact,}
  \text{  regular and polynomially convex}\}.
\end{align*}
 For $E, F\in {\R}\ $ Klimek defined in \cite{Kl95} their distance
\begin{equation}\label{e:Gamma}
  \Gamma
(E,F):=\sup_{z\in\cj}|g_E(z)-g_F(z)|=\max\left(\sup_{z\in E} g_F(z) , \sup_{z\in F} g_E(z) \right)
\end{equation}
and showed that $({\R},\Gamma)$ is a complete metric space. Note that a sequence $(E_n)_{n=1}^\infty$ is convergent to $F$ in $({\R},\Gamma)$ if and only if $g_{E_n}\rightrightarrows g_F$, i.e. the function sequence $(g_{E_n})_{n=1}^\infty$ is uniformly convergent to $g_F$  in the whole complex plane.

 Fix now a polynomial $P$ of degree $d\geq 1$ and consider the
 mapping
\begin{equation}\label{e:A_P}
  A_P: \R\ni K\longmapsto P^{-1}(K)\in \R.
\end{equation}
(\ref{e:Greenaprzeciwobrazu}) yields that $A_P$
is  an isometry if $d=1$ (since $P$ is bijective) and a  contraction with contraction ratio $1/d$ if $d\geq 2$.  In the latter case,  since $(\R,\Gamma)$ is a  complete metric space, by Banach Contraction Principle
$A_P$ has a unique fixed point. This fixed point is the above defined (Definition \ref{def:filledJulia})  filled Julia set $ \K[P]$ (see \cite{Kl95}). In particular $\K[P]\in \R$. Moreover, by the classical proof of Banach Contraction~~Principle
$$
\forall E\in\R:\quad  \lim_{n\rightarrow \infty} P^{-n}(E)=\lim_{n\rightarrow \infty}(A_P)^n(E)= \K[P].
$$
Once again using (\ref{e:Greenaprzeciwobrazu}) we deduce that
$
\forall E\in\R:\; \frac1{d^n}g_E\circ P^n \rightrightarrows g_{\K[P]}.
$

Some properties of the filled Julia set of a polynomial of degree at least 2 (see Definition \ref{def:filledJulia}) followed from the Banach Contraction Principle. We will use a generalization of this result.

\begin{theorem}[Enhanced version of Banach's Contraction Principle, {{\cite[Lemma 4.5]{KK03}}}]\label{th: EnhancedBCP}
  Let $(X, \rho)$ be a
complete metric space and let $(H_n)_{n=1}^\infty$ be a sequence of contractions of X with contraction
ratios not greater than $L<1$. If \;
$
\forall x\in X:\; \sup_{n\geq 1}\rho(H_n(x),x)<\infty,
$
then there exists a unique point $c\in X$ such that the sequence $(H_1\circ...\circ H_n)_{n=1}^\infty$ converges pointwise to $c$.
\end{theorem}

Let us also quote the following result.

\begin{proposition}[{{\cite[Proposition 1]{Kl01}}}]\label{p:nagoya}
Let $P_n:\cj\longrightarrow \cj$ be a  polynomial of degree $d_n\geq 2$ for $n\in\{1,2,...\}$. Let $E\in\R$ and define $E_n:=(P_n\circ...\circ P_1)^{-1}(E)$ for $n\in\{1,2,...\}$. If
\begin{equation}\label{e:sum}
  \sum_{n=1}^\infty \frac{\Gamma(P_{n+1}^{-1}(E), E)}{d_1d_2\cdots d_n}<\infty,
\end{equation}
then the sequence $(E_n)_{n=1}^\infty$ is convergent in $(\R,\Gamma)$ to a set $F$. Any other choice of $\widetilde{E}\in \R$ for which $(\ref{e:sum})$ is satisfied, results in the same limit $F$.   If we assume that $P_n^{-1}(E)\subset E$ for all $n$, then the sequence $(E_n)_{n=1}^\infty$ is decreasing and
 \begin{equation*}
  F=\bigcap_{n\geq 1} E_n=\{z\in E: \; (P_n\circ...\circ P_1)(z)\in E \text{ for all }n\geq 1\}.
 \end{equation*}
\end{proposition}

 Now we would like to use the Klimek metric in our case of guided sequences of polynomials.
The following proposition is an important step in proving that the compact   set  $\K[(p_{n})_{n=1}^\infty]$ obtained in Theorem \ref{thm:JuliadlaKW} is regular and polynomially convex. We
assume additionally that (\ref{e:finite}) holds.

\begin{proposition}\label{prop:GammaEiprzeciwobrazu}
  Let  $(p_n)_{n=1}^\infty$ be a  guided sequence of  polynomials    satisfying $(\ref{e:finite})$.
  If $E\in\R$, then
  $$\exists C>0\ \forall n\geq 1: \quad \Gamma(E,p_n^{-1}(E))\leq C.$$
\end{proposition}
\begin{proof}
Fix $E\in\R$  and $\varepsilon>\max\{0,\log {\rm cap}(E)\}$. Take $\varrho$ from Corollary \ref{cor:przeciwobrazkola}, fix $R\geq
\varrho
$ big enough to satisfy $E\subset \D(0,R)$
 and such that
\begin{equation}\label{e:Rvarepsilon}
  |z|> R\quad \Longrightarrow \quad g_E(z)\leq \log|z|- \log {\rm cap}(E)+\varepsilon
\end{equation}
(this is possible in view of Definitions \ref{defin: Greenfunction} and \ref{def:g_K}). By Corollary \ref{cor:przeciwobrazkola}   $$p_n^{-1}(E)\subset p_n^{-1}\left(\D(0,R)\right)\subset \D(0,R) \quad \text{ for } n\geq 2.$$ Hence
\begin{equation}\label{e:C_1}
  \forall n\geq 2:\qquad \sup_{z\in p_n^{-1}(E)}g_E(z)\leq \sup_{z\in \D(0,R)} g_E(z)=:C_1.
\end{equation}
Note that $C_1$ is a non-negative number    and does not depend on $n$.

The sequence $(p_n)_{n=1}^\infty$ satisfies  (\ref{e:finite}), hence by Proposition \ref{prop: boundedabove}, $$\exists C_2\ \forall z \in E \ \forall n \geq 1:\quad  \frac{1}{\deg p_n}\log |p_n(z)| \leq \frac{1}{\deg p_n}\log^+ |p_n(z)| \leq C_2.$$ Therefore
$
\forall n\geq 1 \ \forall z\in E
: \; |p_n(z)|\leq C_3^{\deg p_n},
$
where $C_3=\exp(C_2)>1$.
In view of (\ref{e:Greenaprzeciwobrazu}) we have
\begin{align}\nonumber
\forall n\geq 1: \quad& \sup_{z\in E}g_{p_n^{-1}(E)}(z) =\sup_{z\in E} \frac{1}{\deg p_n}g_E(p_n(z))\leq\\&\leq \sup_{z\in \D(0,C_3^{\deg p_n})}\frac{1}{\deg p_n}g_E(z)=\sup_{z\in \partial\D(0,C_3^{\deg p_n})}\frac{1}{\deg p_n}g_E(z), \label{e:szacowanieGreenaprzeciwobrazu}
\end{align}
and the last equality follows from the Maximum Principle.

 Consider the set
\[
\mathcal{N}:=\{n \in \mathbb{N}: \; C_3^{\deg p_n} \leq R\}.
\]
By the non-negativity of $g_E$ and the definition of $C_1$  in (\ref{e:C_1})
$$
\forall n\in \mathcal{N} \ \forall  \zeta\in \partial\D\left(0,C_3^{\deg p_n}\right):\quad \frac{1}{\deg p_n} g_E(\zeta)\leq g_E(\zeta)\leq
C_1.
$$
 Combining this with (\ref{e:szacowanieGreenaprzeciwobrazu}) gives
\begin{equation}\label{e:C1bis}
\forall n\in \mathcal{N}: \quad \sup_{z\in E}g_{p_n^{-1}(E)}(z) \leq C_1.
\end{equation}
On the other hand, (\ref{e:Rvarepsilon}) implies
$$
\forall n\notin \mathcal{N} \   \forall  \zeta\in \partial\D(0,C_3^{\deg p_n}):\quad   g_E(\zeta)\leq \deg p_n  \log C_3 - \log {\rm cap}(E)+\varepsilon.
$$
 Recall that $\log C_3=C_2>0$ and $\varepsilon-\log{\rm cap}(E)>0$ by the choice of $\varepsilon$. 
Combining this with (\ref{e:szacowanieGreenaprzeciwobrazu}) gives
\begin{align}\nonumber
  \forall n \notin \mathcal{N}: \quad \sup_{z\in E}g_{p_n^{-1}(E)}(z) &\leq   C_2+\frac1{\deg p_n}(\varepsilon-\log{\rm cap}(E)) \\&  \leq
   C_2+\varepsilon -\log {\rm cap}(E)=: C_4.  \label{e:C_5}
\end{align}
 Note that $C_4>0$ by the choice of $\varepsilon$.

It follows from (\ref{e:C_1}), (\ref{e:C1bis}) and (\ref{e:C_5}) that $
 \forall n\geq 1: \; \Gamma(E,p_n^{-1}(E))\leq \max\{C_1, C_4\}$.
 \end{proof}


\subsection{Julia sets  for guided sequences}

We will now apply the  results from Subsection \ref{s:Klimek}
to a sequence of contractions  
defined in (\ref{e:A_P}).

\begin{theorem}\label{thm:zbiornieautonomiczny}
  Let
  $(p_n)_{n=1}^\infty$ be a  guided sequence of polynomials  satisfying $(\ref{e:finite})$.
   Then  the sequence $$\left(A_{p_{1}}\circ...\circ A_{p_{n}}\right)_{n=1}^\infty$$  converges pointwise   in $(\mathcal{R},\Gamma)$  to a constant mapping with the value
  $\K[(p_{n})_{n=1}^\infty]$. In particular this   limit  set is
polynomially convex and regular.
\end{theorem}
\begin{proof}
   By Theorem \ref{th: EnhancedBCP} and Proposition \ref{prop:GammaEiprzeciwobrazu},  for every $E \in \mathcal{R}$  the sequence   $ \left(A_{p_{2}}\circ...\circ A_{p_{n}}(E)\right)_{n=2}^\infty$  is convergent to the same  set $F\in\R$.

   Take $\varrho>0$ from Corollary \ref{cor:przeciwobrazkola}. Proposition \ref{p:nagoya} yields
   $$\forall R\geq \varrho:\quad  F=  \bigcap_{k=1}^\infty (p_{n}\circ...\circ p_{2})^{-1}\left(\D(0,R)\right) .$$
 Hence $F=\K[(p_{n})_{n=2}^\infty]$ by Theorem \ref{thm:JuliadlaKW}.

 If $\deg p_1>1$, in the same way (renumbering the sequence) we obtain that $ \left(A_{p_{1}}\circ...\circ A_{p_{n}}(E)\right)_{n=1}^\infty$   is convergent to   $\K[(p_n)_{n=1}^\infty] $.

If $\deg p_1=1$, the mapping $A_{p_1}$ is an isometry, hence the convergence of $ \left(A_{p_{2}}\circ...\circ A_{p_{n}}(E)\right)_{n=2}^\infty$  to  $\K[(p_{n})_{n=2}^\infty]$ implies the convergence of $ \left(A_{p_{1}}\circ...\circ A_{p_{n}}(E)\right)_{n=1}^\infty$  to $p_1^{-1}(\K[(p_{n})_{n=2}^\infty])=\K[(p_n)_{n=1}^\infty]$,  this last equality follows from   Remark \ref{rem:pierwszy}.
\end{proof}

\begin{cor}\label{cor:zbieznoscjednostajna}
    Let
  $(p_n)_{n=1}^\infty$ be a  guided sequence of polynomials   satisfying $(\ref{e:finite})$.
   Then
   $\ \forall E\in \R:\; g_{(p_{n}\circ...\circ p_{1})^{-1}(E)}\rightrightarrows g_{\K\left[(p_{n})_{n=1}^\infty\right]}.$

In particular the function sequence
$$\left(\frac1{ \deg p_{n}\cdot...\cdot \deg p_{1}}\log^+\left|p_{n}\circ...\circ p_{1}\right|\right)_{n=1}^\infty$$
is uniformly convergent in $\mathbb{C}$.
\end{cor}
\begin{proof}
  It follows directly from Theorem \ref{thm:zbiornieautonomiczny}
  and the definition of Klimek's metric  (\ref{e:Gamma}). The last assertion is a consequence of (\ref{e:Greenaprzeciwobrazu}) and the formula for the Green function of the unit disk (cf. Example \ref{ex: koloGreen}).
 \end{proof}

 Let us note in passing that if $f_n:\cj\ni z\longmapsto z^2+c_n\in\cj$ for $n\geq 1$ with $(c_n)_{n=1}^\infty \subset \D(0,\frac14)$, then $(f_n)_{n=1}^\infty$ is guided and $g_{\K[(f_n)_{n=1}^\infty]}$ is H\"older continuous (cf. \cite{P}).

The following approximation of the non-autonomous Julia set by the autonomous Julia sets of compositions can be easily shown (cf. \cite[Proposition 5]{AKK}).
\begin{cor}
    Let
  $(p_n)_{n=1}^\infty$ be a  guided sequence of polynomials     satisfying $(\ref{e:finite})$.
 Then
  $ \lim_{k\rightarrow \infty} \Gamma\left(\K[p_{k}\circ...\circ p_{1}], \K[(p_{n})_{n=1}^\infty]\right)=0.$
\end{cor}

  \begin{remark}
  Assumption (\ref{e:finite})  cannot be dropped in the results above. The sequence  $(p_n)_{n=1}^\infty$ from Example \ref{ex:guidedinfty}
does not satisfy this condition and the non-autonomous Julia set associated with  it
is $\{0\} \notin \mathcal{R}$.
\end{remark}

We recall now another result due to Klimek.

\begin{theorem}[{{\cite[Corollary 5]{Kl95}}}]\label{th:Klimekcap}
    $$\forall E,F\in \R:\quad \left|\log{\rm cap} (E)-\log{\rm cap}(F)\right|\leq \Gamma(E,F).$$ In particular the logarithmic capacity is continuous on $(\R,\Gamma)$.
\end{theorem}

Recall that ${\rm cap}\left(\D(0,1)\right)=1$, moreover because of (\ref{e:Greenaprzeciwobrazu}) we have
\begin{equation}\label{e:capprzeciwobrazukola}
    {\rm cap}\left(f^{-1}\left(\D(0,1)\right)\right)=1
    \end{equation}  for any non-constant monic polynomial $f$ too.

\begin{cor}\label{cor:1}
      Let  $(p_n)_{n=1}^\infty$ be a  guided sequence of monic polynomials   satisfying $(\ref{e:finite})$.
   Then
$
{\rm cap}\left(\K[(p_{n})_{n=1}^\infty]\right)=1$.
\end{cor}
\begin{proof}
  In view of Theorem \ref{thm:zbiornieautonomiczny},
  $$ \Gamma\big(\left(p_{n}\circ...\circ p_{1}\right)^{-1}\left(\D(0,1)\right), \K\left[(p_{n})_{n=1}^\infty\right]\big) \longrightarrow 0 \quad \text{  as }\quad n\rightarrow \infty.$$ The assertion follows from (\ref{e:capprzeciwobrazukola}) and Theorem \ref{th:Klimekcap}.
\end{proof}


\subsection{The case of Chebyshev polynomials on   $[-1,1]$  (cf. Main Example)
} \label{s: toy}

\begin{example}\label{ex:rysunki}
The following pictures (prepared by Maciej Klimek) show  approximations of $\K[(t_n)_{n=1}^\infty]$, where $(t_n)_{n=1}^\infty$ is the sequence of minimal polynomials on $[-1,1]$ (see Example \ref{ex:Td}). The sets depicted here are, from left to right:
\begin{itemize}
\item $(t_8\circ...\circ t_2\circ t_1)^{-1}([-1,1]\times [-0.0005,0.0005])$, (which is   used as  an approximation of $(t_8\circ...\circ t_2\circ t_1)^{-1}([-1,1])$),
\item $(t_5\circ...\circ t_2\circ t_1)^{-1}\left(\D(0,1)\right)$, \item
$(t_{100}\circ...\circ t_2\circ t_1)^{-1}\left(\D(0,1)\right)$.
\end{itemize}

\includegraphics[width=4 cm]{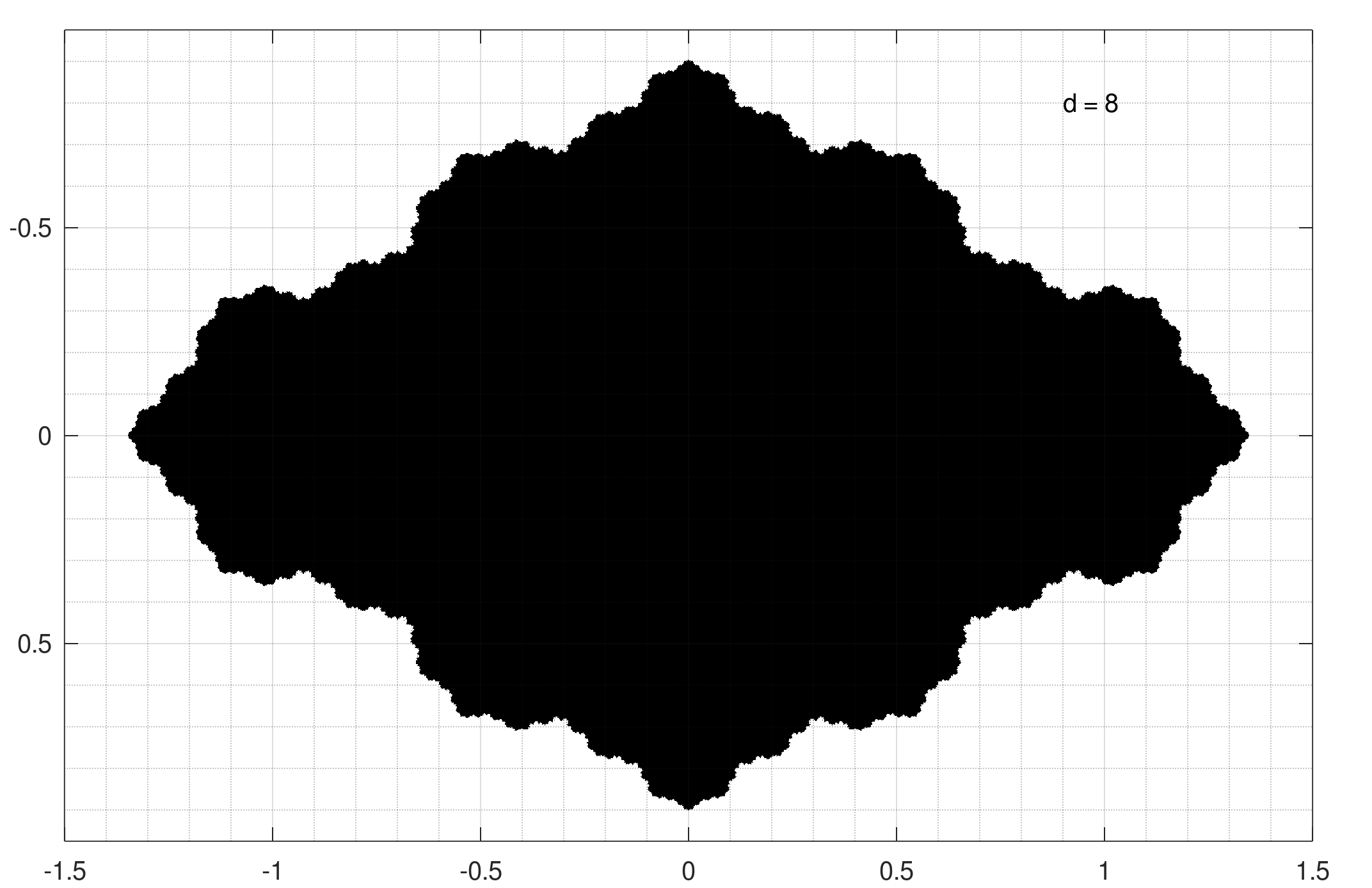}\hfill \includegraphics[width=4 cm]{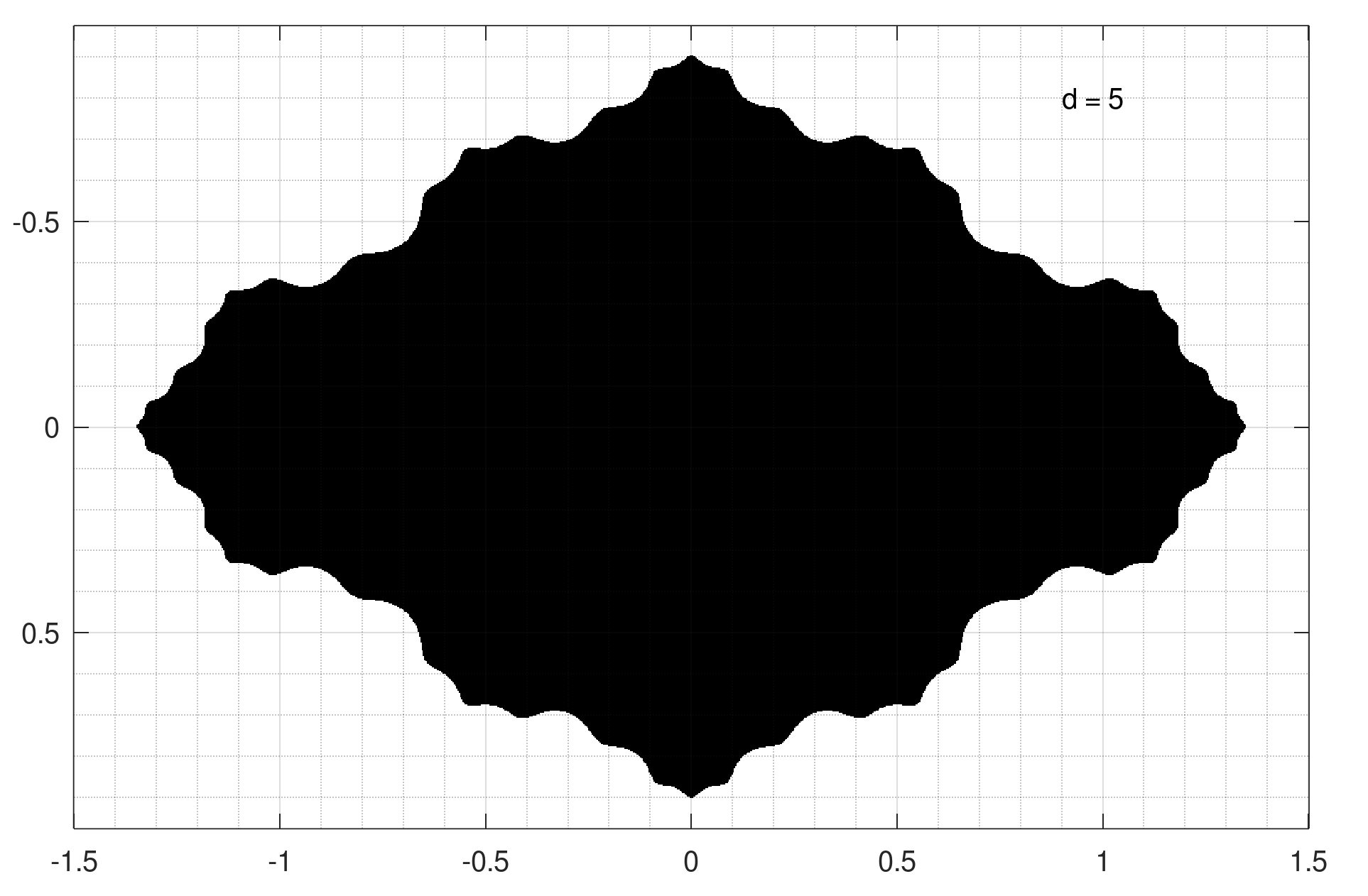}\hfill \includegraphics[width=4 cm]{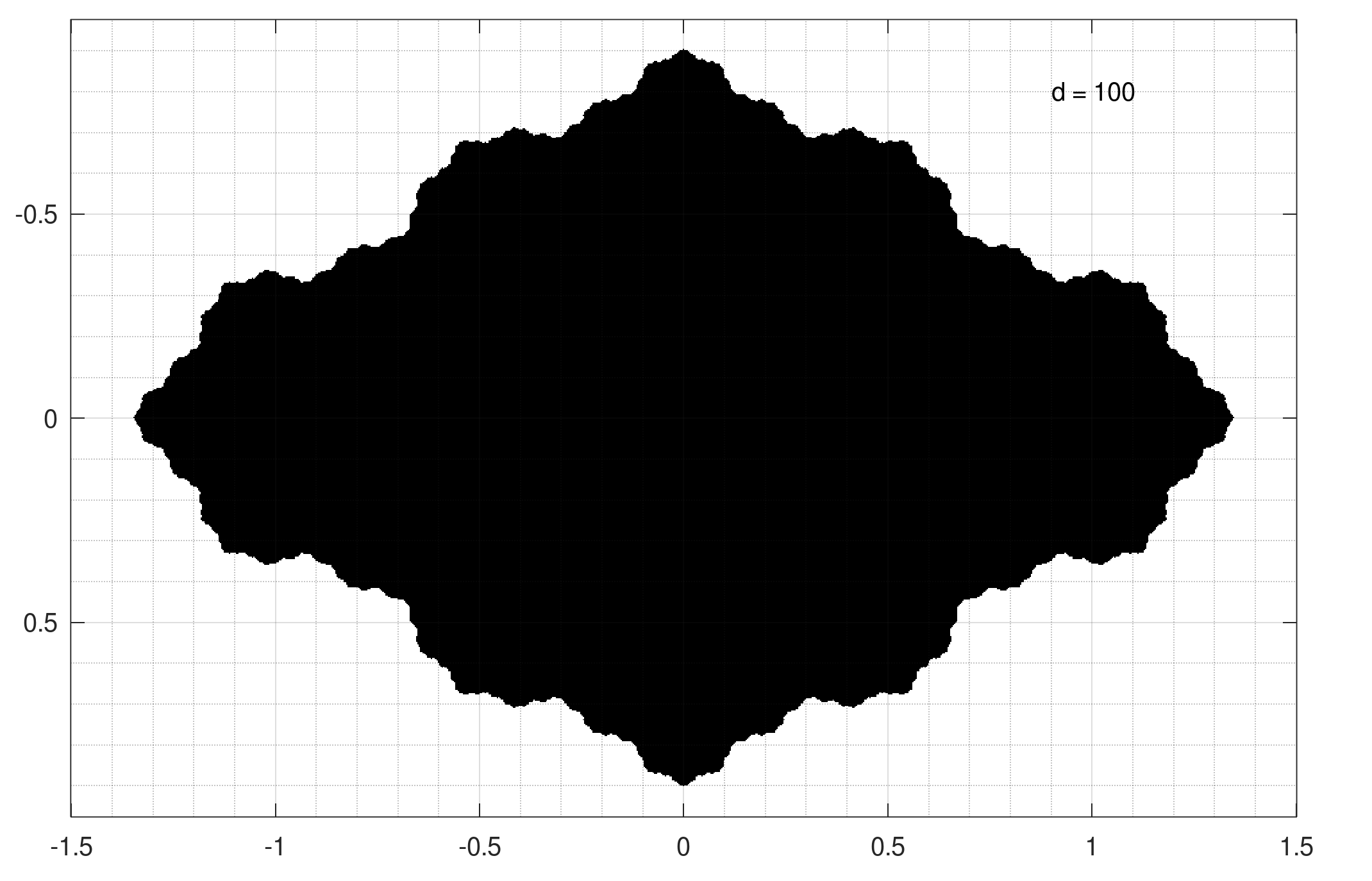}

  Observe some simple geometric  properties of the set $\K[(t_n)_{n=1}^\infty]$.

\underline{Property 1:} $ [-1,1] \subset \K[(t_n)_{n=1}^\infty]$.  In particular, $\K[(t_n)_{n=1}^\infty] \cap  \rj
\neq \emptyset$.
\begin{proof} The set $I=[-1,1]$ is totally invariant under every polynomial $T_n$. Hence $t_1(I)=T_1(I)=I$, $t_2(I)=(1/2)T_2(I)=[-1/2,1/2] \subset I$, $t_n(I)=(1/2^{n-1})T_n(I)=[-1/2^{n-1},1/2^{n-1}] \subset I$, and  $(t_n\circ...\circ t_2\circ t_1)(I) \subset I$ for every $n \geq 1$.
\end{proof}

\underline{Property 2:} If $z \in \K[(t_n)_{n=1}^\infty]$, then $\{-z,\bar{z}\} \subset \K[(t_n)_{n=1}^\infty]$.
\begin{proof} We have $t_1(z)=z$ and $t_2(-z)=t_2(z)$ for every $z \in \mathbb{C}$. Hence $(t_n\circ...\circ t_2\circ t_1)(-z)=(t_n\circ...\circ t_2\circ t_1)(z)$ for every $n \geq 2$ and every $z \in \mathbb{C}$. Moreover, all $t_n$ have real coefficients, so  $(t_n\circ...\circ t_1)(\overline{z})=\overline{(t_n\circ...\circ t_1)(z)}$ for every $n \geq 2$ and every $z \in \mathbb{C}$. It follows that each of the sequences $\left((t_n\circ...\circ t_1)(\overline{z})\right)_{n=1}^\infty$ and $\left((t_n\circ...\circ t_1)(-z)\right)_{n=1}^\infty$ is bounded if and only if $\left((t_n\circ...\circ t_1)(z)\right)_{n=1}^\infty$ is.
\end{proof}

\underline{Property 3:} $\exists R\geq 2:\ \K[(t_n)_{n=1}^\infty]\subset E_R$, where $E_R$ is the filled ellipse with foci $\pm 1$ and  semiaxes $a_R=\frac{1}{2}(R+\frac{1}{R}), \ b_R=\frac{1}{2}(R-\frac{1}{R})$.
\begin{proof} When $R >1$, such (filled) ellipses are sublevel sets (see Definition \ref{def:sublevel} and  Examples \ref{ex: odcGreen} and \ref{ex: elipsyGreen}) of the Green function $g_I$, which tends to infinity as $|z| \to \infty$. Hence we have $\mathbb{C}=\bigcup_{R>1}E_R$, and, by compactness, $\exists R>1: \K[(t_n)_{n=2}^\infty] \subset E_R$. The capacity of $E_R$ is $(a_R+b_R)/2=R/2$ (see again Example \ref{ex: elipsyGreen}). By monotonicity of capacity and Corollary \ref{cor:1}, we need $R \geq 2$ for the inclusion $\K[(t_n)_{n=2}^\infty]\subset E_R$.
\end{proof}

\underline{Property 4:} $\K[(t_n)_{n=1}^\infty]\subset\!\!\!\!\!/ \ E_2$.
\begin{proof}
Note that $E_2$ has the major semiaxis $a_2=5/4$ and the minor semiaxis $b_2=3/4$.
We need to find a point in $\K[(t_n)_{n=1}^\infty]\setminus E_2$.

Let us start with the following facts
for $n\in\{1,2,...\}$:
\begin{enumerate}
\item[(a)] $\forall x\in \rj:\; t_n(-x)=(-1)^nt_n(x)$;
\item[(b)] $t_n$ is increasing in the interval $(1,+\infty)$;
\item[(c)] $\max_{z \in E_2}|t_n(z)|=|t_n(5/4)|=|t_n(-5/4)|=1+2^{-2n} \leq 5/4$.
\end{enumerate}
(a) and (b) are known. To prove (c), we first compute $\max_{w \in E_2}|T_n(w)|$ (cf. \cite{Faber}, \cite{Gaier}). Recall that the classical Chebyshev polynomials satisfy the relation
$T_n \left(\frac{z+z^{-1}}{2}\right)=\frac{z^n+z^{-n}}{2}$ for $n\in\{1,2,...\}.$
For $z=2e^{i\theta}$ with $\theta \in [0,2\pi)$ we thus have
\begin{align*}
T_n \left(\frac{z+z^{-1}}{2}\right)&=\frac{2^ne^{in\theta}+2^{-n}e^{-in\theta}}{2}\\&=\frac{1}{2}\left((2^n+2^{-n})\cos n\theta +i(2^n-2^{-n})\sin n\theta\right).
\end{align*}
Then\[
\left|T_n  \left(\frac{z+z^{-1}}{2}\right)\right|^2=\frac{1}{4}\left(2^{2n}+2\cos 2n\theta +2^{-2n}\right)
\]
achieves its maximum when $\theta =0$ or $\theta=\pi$. Checking values for the corresponding $z=2$ or $z=-2$ we get  $$\max_{w \in E_2}\left|T_n(w)\right|=\left|T_n\left(\frac{2+2^{-1}}2\right)\right|=\left|T_n\left(-\frac{2+2^{-1}}2\right)\right|=2^{n-1}+2^{-(n-1)}.$$
Hence (c) is proved.

Observe now that Property 1 together with (a), (b) and (c) implies that $[-5/4,5/4] \subset \K[(t_n)_{n=1}^\infty]$. Indeed, for every $n \geq 1$ we have $$t_n\left([-5/4,5/4])=t_n([-5/4,-1]\cup[-1,1]\cup[1,5/4]\right) \subset [-5/4,5/4],$$
consequently $ (t_n\circ...\circ t_1)([-5/4,5/4])\subset [-5/4,5/4]$
and we get the inclusion
$[-5/4,5/4] \subset \K[(t_n)_{n=1}^\infty]$.

Consider the point $z_0=4i/5$, which  does not belong to $E_2$ (since the minor semiaxis of $E_2$ is $b_2=3/4 < 4/5$). Now, $$t_2(z_0)=z_0^2-1/2=-57/50 \in (-5/4,-1),$$ hence,   for every $n \geq  2$ we have $(t_n \circ ...\circ t_2\circ t_1)(z_0) \in [-5/4,5/4]$, and so $z_0 \in \K[(t_n)_{n=1}^\infty]$.
\end{proof}
\end{example}


{\footnotesize \begin{acknowledgements}
 Both authors thankfully acknowledge their participation in Thematic Research Programme ``Modern holomorphic dynamics and related
fields'', Excellence Initiative -- Research University programme at the
University of Warsaw (a mini-semester in spring 2023). The paper was partially written thanks to the  support from the programme.

The second named author  extends her thanks to the Faculty of
Mathematics, Informatics and Mechanics of the University of Warsaw  for supporting her participation in the thematic semester ``Dynamical Systems.
Topological, smooth and holomorphic dynamics, ergodic theory, fractals''
  in Stefan Banach International Mathematical Center at the Institute of Mathematics of the Polish Academy of Sciences in Warsaw (part of ``Simons Semesters in Banach Center: 2020s vision.'') and in the conference ``Complex dynamics: connections to other fields'', at the University of Warsaw Conference Center in Ch\k{e}ciny, Poland, as well as  to
the Chair of Approximation, Institute of Mathematics,  Jagiellonian University,  Krak\'ow, for hosting her in March -- June 2023 during her study leave from the American Mathematical Society.

We would also like to thank  Christian Henriksen for the discusion on the notion of $K$-guided sequences and Maciej Klimek for preparing the figures for us.
\end{acknowledgements}}


\end{document}